\documentclass[a4paper,12pt]{article} 

\usepackage{cmap}					
\usepackage{mathtext} 				
\usepackage[T2A]{fontenc}			
\usepackage[utf8]{inputenc}			
\usepackage[english,russian]{babel}	
\usepackage{geometry}

\usepackage{amsmath,amsfonts,amssymb,amsthm,mathtools,amsthm} 
\usepackage{icomma} 
\usepackage[all]{xy}
\mathtoolsset{showonlyrefs=true} 

\usepackage{euscript}	 
\usepackage{mathrsfs} 
\theoremstyle{definition}
\newtheorem{theorem}{Theorem}
\newtheorem*{lemma}{Composition lemma}
\newtheorem*{lemmas}{Lemma on factor algebras}
\newtheorem*{lemmar}{Extension lemma}

\newtheorem*{theoremr}{Theorem about series}
\newtheorem{definition}{Definition}
\newtheorem*{theor}{Theorem}
\newtheorem*{theorpodprost}{The theorem on large subspaces}

\newtheorem*{theormakarenkohuhro}{Khukhro-Makarenko theorem}
\newtheorem{corollary}{Corollary}

\newtheorem*{sled}{Corollary}

\newtheorem*{theormakarenkoshum}{Makarenko-Shumyatsky theorem}

\DeclareMathOperator{\Ker}{\mathop{Ker}}

\DeclareMathOperator{\End}{\mathop{End}}

\DeclareMathOperator{\codim}{\mathop{codim}}

\DeclareMathOperator{\Aut}{\mathop{Aut}}

\newcommand*{\hm}[1]{#1\nobreak\discretionary{}
{\hbox{$\mathsurround=0pt #1$}}{}}

\usepackage{graphicx}  
\graphicspath{{images/}{images2/}}  
\usepackage{wrapfig} 

\usepackage{array,tabularx,tabulary,booktabs} 
\usepackage{longtable}  
\usepackage{multirow} 

\author{\LaTeX{} в Вышке}
\title{1.2 Математика в \LaTeX}
\date{\today}

\begin{document} 

\begin{center}
	\textbf{Khukhro-Makarenko type theorems for algebras}
\end{center}

\begin{center}
	\ Elizaveta Frolova
\end{center}

\begin{center}
	\textit{ \footnotesize  Faculty of Mechanics and Mathematics \\
		\footnotesize Moscow State University \\
		\footnotesize Moscow 119991, Leninskie gory, MSU \\
		\footnotesize elizaveta.fr@gmail.com}
\end{center} 

\footnotesize 
We prove
some anologues for algebras
of recent group-theoretic results 
(due to 
Khukhro,
Klyachko, 
Makarenko,
Milentyeva,
and
Shumyatsky)
on large characteristic subgroups 
satisfying a given property.
\\

\noindent \textbf{0. Introduction} \\

\begin{theormakarenkohuhro}
	\textit{If a group $G$ contains a finite-index subgroup satisfying 
		an outer commutator identity, $G$ contains a 
		characteristic finite-index subgroup satisfying this identity.} 
\end{theormakarenkohuhro}

An \emph{outer \({\rm or} \ multilinear\) commutator identity} is an
identity of the form \\ $[\dots[x_1,\dots,x_t]\dots]=1$ with some meaningful
arrangement of brackets, where all letters $x_1,\dots,x_t$ are different.
Examples of such identities are solvability, nilpotency,
centre-by-metabelianity, etc. A formal definition looks as follows. Let
$F(x_1,x_2,\dots)$ be a free group of countable rank.  An \emph{outer
	commutator of weight 1} is just a letter $x_i$. An \emph{outer commutator
	of weight $t>1$} is a word of the form
$w(x_1,\dots,x_t)=[u(x_1,\dots,x_r),v(x_{r+1},\dots,x_t)]$, where $u$ and
$v$ are outer commutators of weights $r$ and $t-r$, respectively. An
\emph{outer commutator identity} is an identity of the form $w=1$, where
$w$ is an outer commutator.

In [1], ``multilinear properties" \ were 
introduced and a theorem containing, as 
special cases, all results similar to the Khukhro--Makarenko theorem
was proved.

In [3] and [4], an analogue of the Khukhro--Makarenko theorem for algebras 
was proved:

\begin{theor} 
	\textit{Let $G$ be an algebra 
		(possibly 
		non-associative)
		over a field. If $G$ contains a finite-codimension subspace $N$   
		on which a multilinear identity $w(x_1, \ldots, x_t) = 0$ holds, then 
		$G$ contains a 
		finite-codimension 
		subspace $H$ invariant under all surjective 
		endomorphisms and satisfying the same identity.} 
\end{theor}

The present paper deals with algebras using the approach 
developed in [1] for groups.

Section 1 contains basic definitions and the formulation of the main 
theorem of [1].

In Section 2, we prove the theorem on large subspaces, which is an 
analogue for algebras of Klyachko and Milentyeva's theorem.

The theorem on large subspaces has two interesting corollaries. The first 
one was known earlier [4]. The second 
corollary is a new result.

Section 3 
contains an analogue for algebras of
Makarenko--Shumyatsky's theorem from [5]:

\begin{theormakarenkoshum} 
	\textit{Let $G$ be a locally finite group 
		containing a normal subgroup $N$ of a finite index $n$ that has a 
		normal series} 
	\[ N = N_0 \ge N_1 \ge \ldots \ge N_k = 1 \]
	\textit{such that
		each quotient $N_i/N_{i+1}$ is either locally nilpotent or 
		satisfies an outer commutator law $w_i \equiv 1$. Then $G$ contains a 
		finite-index characteristic subgroup $H$ having 
		a characteristic 
		series} 
	\[ H = H_0 \ge H_1 \ge \ldots \ge H_k = 1 \]
	\textit{such that 
		the quotients $H_i / H_{i+1}$ have the same properties as the quotients 
		$N_i / N_{i+1}$, that is, if $N_i / N_{i+1}$ is locally nilpotent, then 
		so is $H_i / H_{i+1}$, and if $N_i / N_{i+1}$ satisifes an outer 
		commutator identity $w \equiv 1$, then so does $H_i / H_{i+1}$.} 
\end{theormakarenkoshum}

The author thanks A. A. Klyachko for usefull discussions. \\

\noindent \textbf{1. Basic concepts}
\begin{definition}		
	An abstract class of algebras $\mathcal{K}$ is called \textit{radical} if:
	
	1) every ideal of an algebra from $ \mathcal {K} $ lies in $ \mathcal {K} $;
	
	2) if $ A = I_1 + I_2 $ where $I_1, I_2 \triangleleft A$, $I_1, I_2 \in \mathcal{K}$ then $A \in \mathcal{K}$.
	\end{definition}

\begin{definition}
	\textit{The formation (the coradical class)} is an abstract class of algebras $ \mathcal {K} $ such that:
	
	1) the factor algebra (factor module) of an algebra from $ \mathcal {K} $ lies in $ \mathcal {K} $;
	
	2) if $A/I_1 \in \mathcal{K}$ and $A/I_2 \in \mathcal{K}$ then $A/(I_1 \cap I_2) \in \mathcal{K}$. 
\end{definition}

\begin{definition}
\textit{A semilattice} is a partially ordered set $\mathcal{L}$ in which every finite subset $\mathcal{N} \subseteq \mathcal{L}$ has a least upper bound $\sup \mathcal{N} \in \mathcal{L}$. 

\textit{A directed semilattice} is a semilattice which is a downwardly directed partially ordered set; this means that for any finite set $\sup \mathcal{N} \in \mathcal{L}$  there is an element  $\inf \mathcal{N} \in \mathcal{L}$ such that  $\inf \mathcal{N} \le N$ for any $N \in \mathcal{N}$ ($\sup$ is the least upper bound, $\inf$ is some lower bound).

A semilattice $\mathcal{L}$  is called \textit{Noetherian} if all increasing chains in it are terminated, that is, there are no infinite chains of the form  $N_1 < N_2 < \ldots$ where $N_i \in \mathcal{L}$.	

A semilattice  $\mathcal{L}$ is called \textit{a lattice} if every finite subset $\mathcal{N} \subseteq \mathcal{L}$ has a greatest lower bound (which we denote by  $\inf \mathcal{N}$ in this case).
\end{definition}

\begin{definition}
A predicate $\mathcal{P}$  on a semilattice $\mathcal{L}$ is called \textit{(poly) monotone} if, from the validity of the property $\mathcal{P}(N_1, \ldots, N_t)$, where  $N_i \in \mathcal{L}$,  it follows that $\mathcal{P}(N'_1, \ldots, N'_t)$  is also true for any  $N'_i \le N_i$. 

A property $\mathcal{P}$ is called \textit{multilinear} if for any $i$  $\mathcal{P}(N_1, \ldots, N_{i-1}, N'_i, N_{i+1}, \ldots, N_t)$ and \\ $\mathcal{P}(N_1, \ldots, N_{i-1}, N''_i, N_{i+1}, \ldots, N_t)$ are satisfied, then a property \\ $\mathcal{P}(N_1, \ldots, N_{i-1}, \sup(N_i', N_i''), N_{i+1}, \ldots, N_t)$ is satisfied.

A predicate $\mathcal{P}$ on a semilattice $\mathcal{L}$ is called \textit{comonotonic} if from the validity of the property $\mathcal{P}(N_1, \ldots, N_t)$, where  $N_i \in \mathcal{L}$  it follows that $\mathcal{P}(N'_1, \ldots, N'_t)$ is also true for any  $N'_i \ge N_i$. 

A predicate   $\mathcal{P}$ is called \textit{comultilinear} if for any $i$ $\mathcal{P}(N_1, \ldots, N_{i-1}, N'_i, N_{i+1}, \ldots, N_t)$ and $\mathcal{P}(N_1, \ldots, N_{i-1}, N''_i, N_{i+1}, \ldots, N_t)$ are satisfied, then a property \\ $\mathcal{P}(N_1, \ldots, N_{i-1}, \inf(N_i', N_i''), N_{i+1}, \ldots, N_t)$ is satisfied for some lower bound  $\inf(N_i', N''_i)$ (we will use the word colinear if $t = 1$).
\end{definition}

\begin{definition}
	We say that the semigroup of endomorphisms $Ф \subseteq \End \mathcal{L}$ of the semilattice $\mathcal{L}$ preserves the property $\mathcal{P}$ (or the property $\mathcal{P}$ is $Ф$-invariant) if $\mathcal{P}(N_1, \ldots, N_t)$ is followed by $\mathcal{P}(\varphi(N_1), \ldots, \varphi(N_t))$ for any $N_i \in \mathcal{L}$ and $\varphi \in Ф$.
	
	An element $N$ of a semilattice is called $Ф$-invariant, if  $\varphi(N) \le N$ for all $\varphi \in Ф$.
	
\end{definition}

\begin{definition}
Let $\mathcal{L}$  be a Noetherian directed semilattice and  $Ф \subseteq \End \mathcal{L}$ be a semigroup of its endomorphisms. The function $\codim \colon \mathcal{L} \to \mathbb{R}$ is called \textit{a generalized $Ф$-codimension} if it has the following properties:

1) $\codim N_1 \le \codim N_2$ if $N_1 \ge N_2$;

2) $\codim \varphi(N) \le \codim N$ for any $N \in \mathcal{L}$ and $\varphi \in Ф$;

3) $\codim \inf (N_1, N_2) \le \codim N_1 + \codim N_2$ for any $N_1, N_2 \in \mathcal{L}$ and for some lower bound $\inf (N_1, N_2)$;

4) in any family $N \subseteq \mathcal{L}$ there are $r \le \max \limits_{N \in \mathcal{N}} \codim N + 1$ elements $N_1, \ldots, N_r$ such that
\[
\sup \mathcal{N} = \sup (N_1, \ldots, N_r).
\] 	
\end{definition}

Next, we give the main tool from [1]:

\begin{theorem} \ [1] \
\textit{Let $\mathcal{L}$  be a Noetherian directed semilattice, $Ф \subseteq \End \mathcal{L}$ be a semigroup of its endomorphisms and $\mathcal{P}$ be a polymonotone multilinear  $Ф$-invariant predicate on $\mathcal{L}$. Then, if there exists an element $N \in \mathcal{L}$ with the property  $P(N, \ldots, N)$, then there exists an element $H \in \mathcal{L}$ such that}

1) \textit{the element $H$ has the same property: $P(H, \ldots, H)$;}

2) \textit{the element $H$ is $Ф$-invariant;}

3) \textit{if $\varphi(N) \le J$ for any $\varphi \in Ф$ and some $J \in \mathcal{L}$ then $H \le J$;}

4) \textit{if $\mathcal{L}$ is a lattice (that is, every finite set has a greatest lower bound) and $Ф$ consists of endomorphisms of the lattice (that is, maps that commute with the operations of taking the greatest lower bound of finite sets), then $ H $ is contained in the sublattice generated by the set \\ $\{\varphi(N); \varphi \in Ф \}$;}

5) \textit{if $\codim \colon \mathcal{L} \to \mathbb{R}$ is a generalized $Ф$-codimension, then   $\codim H \le  f^{t-1}(\codim N)$, where $f^k(x)$ means the k-th iteration of the function $f(x) = x(x+1)$.}\\
\end{theorem} \newpage

\noindent \textbf{2. The theorem on large subspaces}

By an algebra in the following theorem, we mean an algebra (possibly non-associative) over an arbitrary field.
A characteristic subspace in an algebra is a subspace that is invariant under all automorphisms of this algebra.

The formulation of the next theorem is given in [1]. \\

\begin{theorpodprost}
\textit{	Let $N$ be a subspace of an algebra $G$ and eaither}
	
	-- \textit{$N$ has finite codimension,}
	
	--  \textit{or $ N $ is a left ideal and a factor module of $G / N$ is a Noetherian,}
	
	-- \textit{or $ N $ is a two-sided ideal and the factor algebra $G / N$ satisfies the maximality condition for two-sided ideals.}
	
	\textit{Then $G$ contains the characteristic subspaces $ H_1, H_2, \ldots $ such that}
	
	1. \textit{the subspaces $ H_t $ are contained in the lattice of subspaces of the algebra $G$ generated by the images of $N$ for all possible automorphisms of the algebra $G$.} \\
	
	2. \textit{for any multilinear element $w(x_1, \ldots, x_n)$ (which having degree $n \le t$) of a free (non-associative) algebra, the set  $w(H_t, \ldots, H_t)$  is contained in the sum of a finite number of images of the set $w(N, \ldots, N)$ under automorphisms of the algebra $G$. \\
	($w(H_t, \ldots, H_t) \overset{\text{def}}{=} \{ w(h_1, \ldots, h_t) : h_i \in H_i \}$, $h_i$  are not necessarily different). }\\
	
	3. \textit{if codim is either an ordinary codimension (of subspace of G) or a generalized codimension defined on the lattice of ideals, the factor algebras (factor modules) of which lie in $\mathcal{F}$, then  $\codim H_t \hm{\le} f^{t-1}(\codim N)$ where $f^k(x)$ means the k-th iteration of the function $f(x) = x(x+1)$.}
\end{theorpodprost}

\begin{proof}[Proof]
	We apply Theorem 1. As a lattice $ \mathcal {L} $ we take the lattice of subspaces of $G$ generated by the images of $N$ for all possible automorphisms of $G$. As $Ф$ we take the automorphism group of $G$.
	
	We show that the lattice $ \mathcal {L} $ is Noetherian.
	
	Let $N$ be of finite codimension. All images of $N$ under automorphisms of $G$ have the same codimension (because this is an automorphism). The intersection of a finite number of subspaces of finite codimension will be of finite codimension (because the codimension of the intersection does not exceed the sum of codimensions). It turns out that the entire lattice $ \mathcal {L} $ consists of subspaces of finite codimension. This means that $ \mathcal {L} $ is Noetherian.
	
	Let $N$  be a left ideal and a factor module $G/N$ is Noetherian.  All properties are preserved under the automorphism and if the factor module $ G / N $ is Noetherian, then the image under the automorphism will have the same property.
	
	We verify that if $G/N_1$ and $G/N_2$ are Noetherian, then $G/(N_1 \cap N_2)$ is Noetherian. We consider the homomorphism
	\[
	\begin{aligned}
	G \xrightarrow{\varphi}& \ G/N_1 \oplus G/N_2 \\
	g \mapsto& \ (g+N_1, g+ N_2)
	\end{aligned}
	\]
	$\Ker \varphi = N_1 \cap N_2$. The image of the homomorphism $\varphi$ is isomorphic to the quotient by the kernel of this homomorphism. Hence, $G/(N_1 \cap N_2)$ is a submodule in the direct sum:
	\[
	G/(N_1 \cap N_2) \subseteq G/N_1 \oplus G/N_2.
	\]
	The direct sum of Noetherian modules is Noetherian, the submodule of the Noetherian module is Noetherian. Consequently, $G/(N_1 \cap N_2)$ is Noetherian.
	
	Let $G/N_1$ and $G/N_2$ be Noetherian.  Then we prove that  $G/(N_1 + N_2)$ is Noetherian. We consider the map
	\[
	G/N_1 \to G/(N_1 + N_2)
	\]
	This is a surjective homomorphism, so $G/(N_1 + N_2)$ is a quotient of  $G/N_1$ on the kernel. A factor module of a Noetherian module is Noetherian. Therefore, $G/(N_1 + N_2)$ is Noetherian.
	
	It follows that the lattice $\mathcal{L}$ is Noetherian. Let us prove this by contradiction. We take an infinitely expanding chain in G:
	\[
	N_1 \subset N_2 \subset \ldots \ \  \ (N_i \in \mathcal{L}).
	\]
	Then,
	\[
	\{0 \} \subset N_2/N_1 \subset N_3 /N_1 \subset \ldots
	\]
	This is a chain in $G/N_1$. But $G/N_1$ is Noetherian. Hence, at some point this chain ends:
	\[
	\{0 \} \subset N_2/N_1 \subset N_3 /N_1 \subset \ldots \subset N_k/ N_1 = N_{k+1} / N_1 = \ldots
	\]
	Then, $N_k = N_{k+1} = \ldots$. Let us prove this.
	
	Notice that $N_k \subset N_{k+1} \subset \ldots$.  We show that the inclusion is also true in the opposite direction. For any element   $x \in N_{k+1}$  there exists an element $y \in N_k$ such that $x+ N_1 = y + N_1$ (because $N_k/ N_1 = N_{k+1} / N_1$). Then $x = y + n_1$, where $n_1 \in N_1$. But $y \in N_k, \ n_1 \in N_1 \subset N_k$, so $x \in N_k$. 
	
	Thus, the lattice $\mathcal{L}$ is Noetherian. 
	
	Similarly, if $N$ is a two-sided ideal and the factor algebra $ G / N $ satisfies the maximality condition for two-sided ideals, then the lattice $ \mathcal {L} $ is Noetherian. \\
	
As $P(N_1, \ldots, N_t)$ we take the property: \\
	
	for any multilinear element $w$ of a free (non-associative) algebra of degree at most $t$, the set $w(N_1, \ldots, N_n)$  is contained in the sum of a finite number of images of the set $w(N, \ldots, N)$ under automorphisms of the algebra $G$.
	
	Let us prove that the property $P$ is multilinear, that is, that if $P(N'_1, \ldots, N_t)$ holds and $P(N''_1, \ldots, N_t)$ holds, then it follows that $P(N'_1 + N''_1, \ldots, N_t)$ (and so for each argument).
	
	The property $P(N'_1, \ldots, N_t)$ is:
	\[
	w(N'_1, \ldots, N_n) \subseteq \varphi_1 (w(N, \ldots, N)) + \ldots + \varphi_k(w(N, \ldots, N)), \ \ \ (1)
	\]
	where $\varphi_i \in \Aut G$, $w$  is any multilinear element of a free (non-associative) algebra of degree at most $t$.
	
	The property $P(N''_1, \ldots, N_t)$ is:
	\[
	w(N''_1, \ldots, N_n) \subseteq \varphi_1 (w(N, \ldots, N)) + \ldots + \varphi_l(w(N, \ldots, N)), \ \ \ (2)
	\]
	$\varphi_i \in \Aut G$.
	
	The element $w$ is multilinear, then
	\begin{eqnarray}
	w(N'_1 + N''_1, \ldots, N_n) = w(N'_1, \ldots, N_n) + w(N''_1, \ldots, N_n) \overset{(1), (2)}{\subseteq} \\ \subseteq  \varphi_1 (w(N, \ldots, N)) + \ldots + \varphi_k(w(N, \ldots, N)) + \ldots + \varphi_l(w(N, \ldots, N)).
	\end{eqnarray}
	Hence, $P$ is a multilinear property. \\
	
	Let us prove that the property $P$ is monotone, that is, that if the property $P(N_1, \ldots, N_t)$ holds and $N_1 \supseteq \tilde{N}_1$ is a subspace,  then the property $P(\tilde{N}_1, \ldots, N_t)$ holds (and so for each argument). 
	
	The property $P(N_1, \ldots, N_t)$ means that
	\[
	w(N_1, N_2, \ldots, N_n) \subseteq \varphi_1 (w(N, \ldots, N)) + \ldots + \varphi_k(w(N, \ldots, N)).
	\]
	
	$N_1 \supseteq \tilde{N}_1$, then \[
	w(\tilde{N}_1, N_2, \ldots, N_n) \subseteq w(N_1, N_2, \ldots, N_n).
	\]
	Then we get:
	\begin{multline}
			w(\tilde{N}_1, N_2, \ldots, N_n) \subseteq w(N_1, N_2, \ldots, N_n) \subseteq \\ \subseteq \varphi_1 (w(N, \ldots, N)) + \ldots + \varphi_k(w(N, \ldots, N)).		
	\end{multline}

	Consequently, the property $P$ is monotone.\\
	
	Let us prove that the property $P$ is invariant under the authoromorphisms of the algebra $G$. It is necessary to show that if
	\[
	w(N_1, \ldots, N_n) \subseteq \varphi_1 (w(N, \ldots, N)) + \ldots + \varphi_k(w(N, \ldots, N)), \ \ \ \ \ \ \ \ \ \ \ \   (*) 
	\]
	then
	\[
	w( \varphi(N_1), \ldots, \varphi(N_n)) \subseteq \tilde{\varphi}_1 (w(N, \ldots, N)) + \ldots + \tilde {\varphi}_k(w(N, \ldots, N)). \ \ \ (**)
	\]
	
	We apply $ \varphi $ to both sides of (*):\[
	\varphi(w(N_1, \ldots, N_n)) \subseteq \varphi( \varphi_1 (w(N, \ldots, N))) + \ldots + \varphi( \varphi_k(w(N, \ldots, N))).
	\]
	Those, $\tilde{\varphi}_i = \varphi \circ \varphi_i$. And $\varphi(w(N_1, \ldots, N_n)) = w( \varphi(N_1), \ldots, \varphi(N_n)) $, since $\varphi$  is an automorphism  ($\varphi$ of the sum is the sum, $\varphi$ of the product is the product), so $\varphi$ is carried to each term. Therefore, (**) is satisfied. \\
	
	The conditions of the Theorem 1 are satisfied, so the conclusions are also satisfied.
	
We denote by $P_1$ the property $P(N_1, \ldots, N_t)$	 for $t=1$, by $P_2$ the property $P(N_1, \ldots, N_t)$	 for $t=2$, etc. We use Theorem 1. We take our original $N$, take $P$ as the property of $P_1$, and, by Theorem 1, we obtain an element  $H$, which we denote by $H_1$. The subspace  $H_1$  is a characteristic subspace (this follows from part 2) of the Theorem 1).

Then we take the property $ P_2 $, by Theorem 1 we obtain the element $ H_2 $, etc.

Thus, we have found the characteristic subspaces $H_1, H_2, \ldots$.

The item 1. of the theorem on large subspaces follows from part 4) of the Theorem 1.

The item 2. of the theorem on large subspaces follows from part 1) of the Theorem 1.

The item 3. of the theorem on large subspaces follows from part 5) of the Theorem 1.

The theorem is proved.
\end{proof} 

\begin{corollary} \ [4] \
	\textit{Let $G$ be an algebra. If $G$ contains a subspace $N$ of finite codimension on which the multilinear identity  $ w (x_1, \ldots, x_n) = 0 $ (where $ x_1, \ldots, x_n \ in N $) holds, then $G$ contains a subspace $H$ (of finite codimension), which is invariant with respect to all automorphisms of the algebra that satisfies the same identity.} \\
\end{corollary}

\begin{proof}[Proof]
	We set $H = H_t$, where $t = n$ is the degree of the element $w$ ($H_t$ is from the theorem on large subspaces).
	
	By the theorem on large subspaces, for any multilinear element $w(x_1, \ldots, x_n)$ (of degree  $n \le t$) of a free (non-associative) algebra, the set $w(H_t, \ldots, H_t)$ is contained in the sum of a finite number of images of the set $w(N, \ldots, N)$ under automorphisms of the algebra $G$. The zero image is zero under the automorphism, the sum of a finite number of zeros is also zero.  Therefore, $w(H_t, \ldots, H_t) = 0$. 	
\end{proof}

\begin{corollary}
	\textit{Let $G$ be an algebra, $N$ be a subspace of $G$ of finite codimension, and $ w $ be a multilinear element of a free (non-associative) algebra. Then if  $w(N, \ldots, N)$ is contained in some finite-dimensional subspace, then $G$ contains a subspace $ H $ of finite codimension that is invariant under all automorphisms of the algebra such that $w(H, \ldots, H)$ will also be contained in finite-dimensional subspace.}\\
\end{corollary}

\begin{proof}[Proof]
	We set $H = H_t$, where $t = n$ is the degree of the element  $w$ ($H_t$ from the theorem on large subspaces).
	
	If $w(N, \ldots, N)$ is contained in some finite-dimensional subspace, then $w(H_t, \ldots, H_t)$ will also be contained in the finite-dimensional subspace since image under the automorphism of a finite number of finite-dimensional spaces is a finite-dimensional space.
\end{proof} 
\newpage

\noindent \textbf{3. An analogue for algebras of Makarenko-Shumyatsky's theorem}

\begin{lemma} \ [1] \
\textit{Suppose that on some lattice there is a multilinear monotone predicate $\mathcal{Q}(M_1, \ldots, M_k)$ and a set of predicates \[\mathcal{R} = \left\{ \mathcal{R}_i \begin{pmatrix}
X_1, \ldots, X_l\\
Y
\end{pmatrix} \right\},\] which on the first line (that is, for any fixed second line) are multilinear and monotone, and on the second line (that is, for any fixed first line)  are colinear and comonotonic. Then the predicate
\begin{multline}
\mathcal{Q} \circ \mathcal{R}(N_1, \ldots, N_{kl}) = \\ = \left( \exists M_1, \ldots, M_k \ \ \mathcal{Q}(M_1, \ldots, M_k) \ \text{and for} \ i \in \{1, \ldots, k\} \ \mathcal{R}_i  \begin{pmatrix}
N_{(i-1)l+1}, \ldots, N_{li}\\
M_i
\end{pmatrix} \right)
\end{multline}\\
called the composition of predicates $\mathcal{Q}$ and $\mathcal{R}$ is multilinear and monotone.}
\end{lemma}

The proof of this lemma is given in [1]. \\

The following two lemmas are analogues (for algebras) of lemmas for groups stated in [1]. \\

\begin{lemmas}
\textit{Let G be an algebra, N, M be ideals in G. Each of the following two properties $\mathcal{A}(N,M)$ and $\mathcal{B}(N,M)$ can be written in the form $\mathcal{R}\begin{pmatrix}
N, \ldots, N \\
M
\end{pmatrix}$ where the predicat $\mathcal{R}$ on the lattice of ideals is monotone and multilinear on the first line, and on the second line it is comonotonic and colinear:}
\[
\begin{aligned}
\mathcal{A}(N,M) = \left( N/(N \cap M) \ \text{\textit{satisfies (fixed) multilinear identity}} \ w = 0 \right), \\
\mathcal{B}(N,M) = \left( N/(N \cap M) \ \text{\textit{belongs to a (fixed) radical formation}} \ \mathcal{F}\right).
\end{aligned}
\]

\end{lemmas}

\begin{proof}[Proof]

Let us prove that for a property $\mathcal{A}$ the predicate
\[
\mathcal{R}\begin{pmatrix}
N_1, \ldots, N_t \\
M
\end{pmatrix} = \left( w(N_1, \ldots, N_t) \subseteq M \right)
\]
satisfies all conditions (where $N_1, \ldots, N_t, M \triangleleft G$).

Notice that
\[
A(N, M) = \mathcal{R}\begin{pmatrix}
N, \ldots, N \\
M
\end{pmatrix},
\]
since property $A(N, M)$  is that the quotient $N / (N \cap M)$ satisfies (fixed) multilinear identity $w = 0$. This means (by definition) that $w(N, \ldots, N) \subseteq N \cap M$. But in N the set $w(N, \ldots, N) $  is automatically bounded; hence, this is equivalent to the fact that  $w(N, \ldots, N) \subseteq M$.

Let us check the monotonicity of the first line.

The monotonicity of the first line means that if $\mathcal{R}\begin{pmatrix}
N_1, \ldots, N_t \\
M
\end{pmatrix}$ then $\mathcal{R}\begin{pmatrix}
N'_1, \ldots, N_t \\
M
\end{pmatrix}$ where $N'_1 \subseteq N_1$ (and so for each argument). Indeed, the property $\mathcal{R}\begin{pmatrix}
N_1, \ldots, N_t \\
M
\end{pmatrix}$ means that $\left( w(N_1, N_2 \ldots, N_t) \subseteq M \right)$. Since $N'_1 \subseteq N_1$ then $w(N'_1, N_2 \ldots, N_t) \subseteq w(N_1, N_2 \ldots, N_t)$. Hence, \\ $w(N'_1, N_2 \ldots, N_t) \subseteq M$, which means that the property  $\mathcal{R}\begin{pmatrix}
N'_1, \ldots, N_t \\
M
\end{pmatrix}$ is satisfied. Thus the monotonicity property of the first line is proved. \\

Let us check the multilinearity of the first line, i.e. that if the properties
\[
\mathcal{R}\begin{pmatrix}
N_1, \ldots, N'_i, \ldots, N_t \\
M
\end{pmatrix} \ \text{and} \ \mathcal{R}\begin{pmatrix}
N_1, \ldots, N''_i, \ldots, N_t \\
M
\end{pmatrix} \ \ \ \ \ (*)
\]
then the property $\mathcal{R}\begin{pmatrix}
N_1, \ldots, N'_i + N''_i, \ldots, N_t \\
M
\end{pmatrix}$, $\forall \ i$ is satisfied. 

(*) means that $w(N_1, \ldots, N'_i, \ldots, N_t) \subseteq M$ and $w(N_1, \ldots, N''_i, \ldots, N_t) \subseteq M$. Since $w$ a multilinear element, it follows that
\[
w(N_1, \ldots, N'_i, \ldots, N_t) + w(N_1, \ldots, N''_i, \ldots, N_t) = w(N_1, \ldots, N'_i + N''_i, \ldots, N_t) \subseteq M.
\]
Thus the multilinearity of the first line is proved. \\

Let us check the comonotonicity of the second line, i.e. that if the property $\mathcal{R}\begin{pmatrix}
N_1,  \ldots, N_t \\
M
\end{pmatrix}$ then the property $\mathcal{R}\begin{pmatrix}
N_1,  \ldots, N_t \\
M'
\end{pmatrix}$ where $M' \supseteq M$ is satisfied.

Indeed, the property $\mathcal{R}\begin{pmatrix}
N_1,  \ldots, N_t \\
M
\end{pmatrix}$ means that $w(N_1, \ldots, N_t) \subseteq M$ and since \\ $M \subseteq M'$ then $w(N_1, \ldots, N_t) \subseteq M'$ those the property $\mathcal{R}\begin{pmatrix}
N_1,  \ldots, N_t \\
M'
\end{pmatrix}$ is satisfied. \\

Let us check the colinearity of the second line, i.e. that if the properties
\[
\mathcal{R}\begin{pmatrix}
N_1,  \ldots, N_t \\
M_1
\end{pmatrix} \ \text{and} \ \mathcal{R}\begin{pmatrix}
N_1,  \ldots, N_t \\
M_2
\end{pmatrix} \ \ \ \ \ (**)
\]
are satisfied, then the property $\mathcal{R}\begin{pmatrix}
N_1,  \ldots, N_t \\
M_1 \cap M_2
\end{pmatrix}$ is satisfied. \\

(**) means that $w(N_1,  \ldots, N_t) \subseteq M_1$ and $w(N_1,  \ldots, N_t) \subseteq M_2$. Hence, \\ $w(N_1,  \ldots, N_t) \subseteq M_1 \cap M_2$, which means that the property $\mathcal{R}\begin{pmatrix}
N_1,  \ldots, N_t \\
M_1 \cap M_2
\end{pmatrix}$ is satisfied. \\ \\

For the property $\mathcal{B}$ we can take:
\[
\mathcal{R} \begin{pmatrix}
N \\
M
\end{pmatrix} = \mathcal{B}(N, M).
\]
Let's check the linearity of the first line, i.e. that if the properties
\[
\mathcal{R} \begin{pmatrix}
N_1 \\
M
\end{pmatrix} \ \text{and} \ \mathcal{R} \begin{pmatrix}
N_2 \\
M
\end{pmatrix} \ \ \ \ (3)
\]
are satisfied, then $\mathcal{R} \begin{pmatrix}
N_1 + N_2 \\
M
\end{pmatrix}$ is satisfied. \\

(3) means that $N_1/(N_1 \cap M) \in \mathcal{F}$ and $N_2/(N_2 \cap M) \in \mathcal{F}$. It is necessary to show that then $ (N_1 + N_2) / ((N_1 + N_2) \cap M) \in \mathcal{F}$.

Notice that
\[
N_1 / ((N_1 + N_2) \cap M \cap N_1) \subseteq (N_1 + N_2) / ((N_1 + N_2) \cap M).
\]
But $(N_1 + N_2) \cap M \cap N_1  = M \cap N_1$, because $N_1 + N_2$ larger than $N_1$. I.e.
\[
N_1 / (M \cap N_1) \subseteq (N_1 + N_2) / ((N_1 + N_2) \cap M).
\]
Similarly,
\[
N_2 / (M \cap N_2) \subseteq (N_1 + N_2) / ((N_1 + N_2) \cap M).
\]
Since the ideals $N_1, N_2$ gave the whole of the algebra in the sum, then their image will give a total of the image, therefore,
\[
(N_1 + N_2) / ((N_1 + N_2) \cap M) = N_1 / (M \cap N_1)  + N_2 / (M \cap N_2).
\]
It is known that  $N_1/(N_1 \cap M) \in \mathcal{F}$ and $N_2/(N_2 \cap M) \in \mathcal{F}$. Then by 2) from the definition of the radical class we obtain that $ (N_1 + N_2) / ((N_1 + N_2) \cap M) \in \mathcal{F}$. \\

Let us check the monotonicity of the first line, i.e. that if the property $\mathcal{R} \begin{pmatrix}
N \\
M
\end{pmatrix}$ is satified, then the property $\mathcal{R} \begin{pmatrix}
N' \\
M
\end{pmatrix}$ where $N' \subseteq N$ is satisfied.

It is known that $N / (N \cap M) \in \mathcal{F}$ and $N' \subseteq N$. Let us prove that then $N' / (N' \cap M) \in \mathcal{F}$. We consider the canonical homomorphism:\[
N \to N / (N \cap M).
\]
Under this homomorphism\[
N' \to N' / (N' \cap N \cap M),
\]
i.e. (because $N' \subseteq N$)
\[
N' \to N' / (N'  \cap M).
\]
Notice that $N' \triangleleft N$ (since $N', N \triangleleft G, N' \subseteq N$). Because under a surjective homomorphism the image of an ideal is an ideal then
\[
N' / (N' \cap M) \triangleleft N/ (N \cap M).
\]
Consequently, by item 1) from the definition of the radical class we obtain that $N' / (N' \cap M) \in \mathcal{F}$. \\

Let us check the colinearity of the second line. Let $N / (N \cap M_1) \in \mathcal{F}$ and \\ $N / (N \cap M_2) \in \mathcal{F}$. Let us prove that then $N / (N \cap (M_1 \cap M_2)) \in \mathcal{F}$. By item 2) of the definition of the coradical class we find that if $N / (N \cap M_1) \in \mathcal{F}$ and \\ $N / (N \cap M_2) \in \mathcal{F}$ then $N / ((N \cap M_1) \cap (N \cap M_2)) \in \mathcal{F}$ but $(N \cap M_1) \cap (N \cap M_2) \hm{=} N \cap (M_1 \cap M_2)$. Consequently, $N / (N \cap (M_1 \cap M_2)) \in \mathcal{F}$. \\

Let us check the comonotonicity of the second line. Let $N / (N \cap M) \in \mathcal{F}$ and $M' \supseteq M$. We prove that then $N / (N \cap M') \in \mathcal{F}$.

We consider the surjective homomorphism:
\[
N / (N \cap M) \to N / (N \cap M').
\]
We get that $N / (N \cap M')$ is the factor algebra of the algebra $N / (N \cap M)$.  By item 1) of the definition of the coradical class we obtain that $N / (N \cap M') \in \mathcal{F}$. 

The theorem is proved.
\end{proof}

\begin{lemmar}
\textit{Let $\mathcal{Q}(M_1, \ldots, M_l)$ is monotonic multilinear predicate on the lattice of ideals of the algebra $G$, $w$ is a multilinear element (of a free algebra) of degree d and $\mathcal{F}$ is a radical formation. Then each of the following two ideal's properties (of the algebra $G$) $\mathcal{C}(N)$ and $\mathcal{D}(N)$ can be written in the form $\mathcal{P}(N, \ldots, N)$, where the predicate $P(N_1, \ldots, N_t)$ on the lattice of ideals is monotone and multilinear and $t = ld$ for the property $\mathcal{C}$ and $t = l$ for the property}\textit{ $\mathcal{D}$:
\[
\begin{aligned}
\mathcal{C}(N) =& \left( \exists M \triangleleft G \ \ M \subseteq N, \ \mathcal{Q}(M, \ldots, M) \ \text{and} \ N/M \ \text{satisfies the identity} \ w = 0 \right), \\
\mathcal{D}(N) =& \left( \exists M \triangleleft G \ \ M \subseteq N, \ \mathcal{Q}(M, \ldots, M) \ \text{and} \ N/M \in \mathcal{F} \right).
\end{aligned}
\] }
\end{lemmar}

\begin{proof}[Proof]
The property $\mathcal{C}(N)$ can be rewritten in the equivalent form:
\[
\mathcal{C}(N) = \left( \exists M \triangleleft G \ \ \mathcal{A}(N,M), \ \text{and} \  \mathcal{Q}(M, \ldots, M) \right),
\]
where $\mathcal{A}$ is from the lemma of factor algebras (if $M \subseteq N$ then $M \cap N = M$).

The last formula is equivalent to this:
\[
\mathcal{C}(N) = \left( \exists M_1, \ldots, M_l \triangleleft G \ \ \mathcal{A}(N,M_1), \ldots, \mathcal{A}(N,M_l) \ \text{and} \  \mathcal{Q}(M_1, \ldots, M_l) \right).
\]

To verify this (in the non-obvious direction) we set $M = \bigcap M_i$. Then the equivalence follows from the fact that the property $\mathcal{Q}$ is monotone  and the variety of algebras given by the identity $ w = 0 $ satisfies the definition of the formation. \\

By the lemma on factor algebras, we can rewrite this property in the form:
\[
\mathcal{C}(N) = \left( \exists M_1, \ldots, M_l \triangleleft G \ \ \ \ \ \mathcal{R}\begin{pmatrix}
N, \ldots, N \\
M_1
\end{pmatrix}, \ldots, \mathcal{R}\begin{pmatrix}
N, \ldots, N \\
M_l
\end{pmatrix} \ \text{and} \  \mathcal{Q}(M_1, \ldots, M_l) \right),
\]
where the predicate $\mathcal{R}$ on the lattice of ideals is monotone and multilinear in the first line and on the second line it is comonotonic and colinear.

The reference to the composition lemma completes the proof. For the property $\mathcal{D}$ the arguments are analogous. \\ 
\end{proof}

\begin{theoremr}
\textit{Let G be an algebra, N be an ideal in G, a factor algebra $G/N$ is finite-dimensional. Let N possesses a series
\[
\{0\} = A_0 \subseteq \ldots \subseteq A_n = N \ \ \ \ (A_i \triangleleft G)
\]
each of whose quotients  $A_i/A_{i-1}$ is either satisfies the multilinear identity $w_i = 0$, or is in the radical class  $\mathcal{K}_i$ and all these classes are also coradical simultaneously. Then $G$ contains an ideal $M$ such that $ G / M $ is finite-dimensional and $M$ is invariant under all automorphisms of $G$, and $M$ has the same property (that is, a series of the same length, each quotient of which satisfies the same identity or lies in the same class that the corresponding quotient of the original series).} \\
\end{theoremr}

\begin{proof}[Proof]
We denote by $ U_n (N) $ the following property:\\

$N$ possesses a series
\[
\{0\} = A_0 \subseteq \ldots \subseteq A_n = N \ \ \ \ (A_i \triangleleft G)
\]
each of whose quotients  $A_i/A_{i-1}$ is either satisfies the multilinear identity $w_i = 0$, or is in the radical class  $\mathcal{K}_i$ and all these classes are also coradical simultaneously. \\

(The predicate $ U_n (N) $ is considered on the lattice of all ideals of finite codimension.) This lattice is Noetherian since if the codimension is finite, then it can not decrease indefinitely, so all increasing chains terminate).

Let us prove that $ U_n (N) $ can be written in the form $ P (N, \ldots, N) $, where $ P $ is a multilinear and monotone predicate. We shall prove by induction on the length of the series (that is, with respect to n).

The base of induction: $ n = 0 $. In the zero algebra everything is satisfied, obviously.

The inductive step. We rewrite the property $ U_n (N) $ in the form:
\begin{multline}
U_n(N) = \left( \exists \ A_{n-1} \triangleleft G, \ A_{n-1} \subseteq N: \  N/A_{n-1} \text{ is either satisfies the identity}  \ w_n = 0, \right. \\ \left. \text{or is in the} \ \mathcal{K}_n; \ \text{and} \ U_{n-1}(A_{n-1}) \right). 
\end{multline}
The property $U_{n-1}(A_{n-1})$ is the same: $A_{n-1}$ possesses a series of smaller lengths, each of whose quotients is either satisfies the identity $w_j = 0$ or is in the class  $\mathcal{K}_j$.

By the  inductive hypothesis, $U_{n-1}(A_{n-1})$ can be written in the form \\ $P'(A_{n-1}, \ldots, A_{n-1})$ where $P'$ is a multilinear and monotone predicate.

Then, by the extension lemma, we obtain that  $U_n(N)$ can be written in the form $P(N, \ldots, N)$ where $P$ is a multilinear monotone predicate of t arguments  (where $t = \prod t_i$, $t_i$ is the degree $w_i$). By Theorem 1 we obtain the assertion of the theorem about series.
\end{proof}

We obtain an analogue (for algebras) of the theorem in [3] (in which we consider the case when the group $G$ is locally finite, and every class $ \mathcal {K} _i $ is the class of all locally nilpotent groups):

\begin{sled}
\textit{Let $G$ be either an associative algebra or a Lie algebra, $N$ be an ideal in $G$, the factor algebra $ G / N $ is finite-dimensional. Let $N$ possesses a series
\[
\{0\} = A_0 \subseteq \ldots \subseteq A_n = N \ \ \ \ (A_i \triangleleft G)
\]
each of whose quotients  $A_i/A_{i-1}$ is either satisfies the multilinear identity $w_i = 0$, or is locally nilpotent. Then $G$ contains an ideal $M$ such that $ G / M $ is finite-dimensional and $M$ is invariant under all automorphisms of $G$, and $M$ has the same property (that is, a series of the same length, each quotient of which satisfies the same identity or is locally nilpotent). }\\
	\end{sled} 

\begin{proof}[Proof]
	Let $ G $ be an associative algebra. We consider the class of all associative locally nilpotent algebras.
	Let us prove that this is a radical class.
	
	The ideal of a locally nilpotent algebra is locally nilpotent (as an algebra), because a locally nilpotent algebra $G$ is an algebra, each finitely generated subalgebra of which is nilpotent. Any finitely generated subalgebra in this ideal will be a subalgebra in the whole algebra $G$. Hence, it will be locally nilpotent. \\
	
	Let $I$ and $J$ be locally nilpotent ideals of the associative algebra $G$. Then $ I + J $ is a locally nilpotent ideal in $G$. We prove this.
	
	By the isomorphism theorem, we have an isomorphism:
	\[
	(I+J) / J \simeq I/(I \cap J)
	\]
	The factor algebra $I/(I \cap J)$ is locally nilpotent as a homomorphic image of a locally nilpotent algebra $I$. Consequently, the factor algebra $ (I + J) / J $ is locally nilpotent.
	
	We have: J is locally nilpotent (by assumption) and $ (I + J) / J $ is locally nilpotent. Let us prove that this implies that $ I + J $ is locally nilpotent (the proof of this fact which has given in [6] for rings is used here for algebras).
	
	We consider $A = \{x_1, x_2, \ldots, x_n \}$, it is a finite subset of $I+J$.  Let $\tilde{A} = \{ \tilde{x}_1, \tilde{x}_2, \ldots, \tilde{x}_n \}$, where $\tilde{x}_i = x_i + J$. The quotient $ (I + J) / J $ is locally nilpotent, then $\exists$ an integer $N$ such that $\tilde{x}_{i_1} \ldots \tilde{x}_{i_N} = 0$ for $i_j = 1,2, \ldots, n$. 
	
	The set $B = \{x_{i_1} \ldots x_{i_N} \ | \ i_j = 1,2,\ldots, n  \}$ is finite and is contained in $J$, and $J$ is locally nilpotent. Consequently, $\exists$ an $М$ such that the product of any M of such elements is 0. This means that the product of any $x_i$, $i = 1, \ldots, MN$  is 0. Consequently, $ I + J $ is locally nilpotent.
	
	Thus, the class of all associative locally nilpotent algebras is a radical class.
	
	Let us prove that this class is also coradical. The factor algebra of a locally nilpotent algebra is locally nilpotent (if the product of some number of elements is zeroed in the algebra itself, then this is also satisfied in the factor algebra).
	
Let's prove that if $A/I_1 \in \mathcal{K}$ and $A/I_2 \in \mathcal{K}$ then $A/(I_1 \cap I_2) \in \mathcal{K}$.
	
	We consider the homomorphism
	\[
	\begin{aligned}
	A \xrightarrow{\varphi}& \ A/I_1 \oplus A/I_2 \\
	a \mapsto& \ (a+I_1, a+ I_2)
	\end{aligned}
	\]
	$\Ker \varphi = I_1 \cap I_2$. The image of the homomorphism $ \varphi $ is isomorphic to the quotient by the kernel of this homomorphism. So 
	\[
	A/(I_1 \cap I_2) \subseteq A/I_1 \oplus A/I_2.
	\]
By hypothesis, $ A / I_1 $ and $ A / I_2 $ are locally nilpotent. We have proved above that the sum of locally nilpotent ideals is locally nilpotent.

Consequently, $ A / (I_1 \cap I_2) $ is locally nilpotent. \\
	
Let $ G $ be a Lie algebra. We consider the class of locally nilpotent Lie algebras. Let us prove that this class is radical.

The item 1) from the definition of the radical class is satisfied.
	
	Let us prove that if $ I, J $ are locally nilpotent ideals of a Lie algebra L, then $ I + J $ is locally nilpotent. It is necessary to show that any finitely generated subalgebra of $ I + J $ is nilpotent. We take a finitely generated subalgebra in $ I + J $, write its generators in the form $ a_i + b_i $, $ i = \overline {1, k} $, where $ a_i \in I, b_i \in J $. Denote $ I_k = \langle a_1, \ldots, a_k \rangle $, $ J_k = \langle b_1, \ldots, b_k \rangle $. $ I_k, J_k $ are finitely generated subalgebras of $ I, J $, respectively, they are nilpotent.
	
	Notice that \[
	\left( I_k + \langle I_k \cdot J_k \rangle \right) + \left( J_k + \langle I_k \cdot J_k \rangle\right) =  \langle I_k, J_k \rangle.
	\]
Denote \[
	U_k = I_k + \langle I_k \cdot J_k \rangle, \ \ V_k = J_k + \langle I_k \cdot J_k \rangle.
	\]
	
	Notice that $U_k \triangleleft \langle I_k, J_k \rangle$, $V_k \triangleleft \langle I_k, J_k \rangle$.
	
	Let us prove that $U_k$ and $V_k$ are nilpotent. Notice that \\ $I \cdot J \subseteq I \cap J$, т.к. $I \cdot J \subseteq I, \ I \cdot J \subseteq J$ (because $I$ and $J$ are ideals). 
	
	Denote $C = \{a_i b_j \ | \ i = \overline{1, k}, \ j = \overline{1, k} \}$. \\
	$\langle C \rangle $ is a finitely generated subalgebra and $\langle C \rangle  \subseteq I$ and $\langle C \rangle  \subseteq J$ (because $I$ and $J$ are ideals).
	
	Notice that $\langle I_k, C \rangle  \subseteq I$ (because $I$ is an ideal), and $\langle I_k, C \rangle$ is a finitely generated subalgebra. Therefore, $\langle I_k, C \rangle$ is nilpotent. But $U_k \subseteq \langle I_k, C \rangle \subseteq I$, which means $U_k$ is nilpotent. 
	
	Similarly, $V_k \subseteq \langle J_k, C \rangle \subseteq J$, then $V_k$ is nilpotent. 
	
	The sum of the nilpotent ideals of a Lie algebra is nilpotent. Thus, the class of locally nilpotent Lie algebras is a radical class. \\
	
	Let us prove that this class is also coradical. The factor algebra of a locally nilpotent algebra is locally nilpotent. 
	
	Let us prove that if  $A/I_1 \in \mathcal{K}$ and $A/I_2 \in \mathcal{K}$, then $A/(I_1 \cap I_2) \in \mathcal{K}$.
	
	We consider the homomorphism
	\[
	\begin{aligned}
	A \xrightarrow{\varphi}& \ A/I_1 \oplus A/I_2 \\
	a \mapsto& \ (a+I_1, a+ I_2)
	\end{aligned}
	\]
	$\Ker \varphi = I_1 \cap I_2$. The image of the homomorphism $\varphi$ s isomorphic to the quotient by the kernel of this homomorphism. Hence
	\[
	A/(I_1 \cap I_2) \subseteq A/I_1 \oplus A/I_2.
	\]
	By hypothesis,  $A/I_1$ and $A/I_2 $ are locally nilpotent. We have proved above that the sum of locally nilpotent ideals is locally nilpotent.
	
	Consequently,  $A/(I_1 \cap I_2)$ is locally nilpotent. Thus, the class of locally nilpotent Lie algebras is a coradical class.
	
\end{proof}

\begin{center}
	\noindent\textbf{REFERENCES} 
\end{center}

[1] \ \ Anton A. Klyachko, Maria V. Milentyeva,  Large and symmetric: The Khukhro–Makarenko theorem on laws — without laws, J. Algebra \textbf{424} (2015), 222-241.

[2] \  E. I. Khukhro, N. Yu. Makarenko, Large characteristic subgroups satisfying multilinear commutator identities, J. London Math. Soc. \textbf{75}, no. 3 (2007), 635–646.

[3] \ E. I. Khukhro, Ant. A. Klyachko, N. Yu. Makarenko, Yu. B. Melnikova, Automorphism invariance and identities, Bull. London Math. Soc. \textbf{41}, no. 5 (2009), 804–816.

[4] \   E. I. Khukhro, N. Yu. Makarenko, Automorphically-invariant ideals satisfying multilinear identities, and group-theoretic applications, J. Algebra \textbf{320}, no. 4 (2008), 1723-1740.

[5] \ N. Yu. Makarenko, P. Shumyatsky, Characteristic subgroups in locally finite groups, J. Algebra \textbf{352}, no. 1 (2012), 354-360.

[6] \  Nathan Jacobson, Structure of rings, American Mathematical Society, 1956.

\end{document}